\documentclass[12pt]{amsart} 

\usepackage{amsmath}
\usepackage{amssymb}
\usepackage{amsthm}
\usepackage{color}
\usepackage[mathscr]{euscript}
\usepackage{verbatim}
\usepackage{graphicx}
\usepackage[margin=1in]{geometry}
\usepackage{multicol}

\newtheorem{theorem}{Theorem}[section]
\newtheorem{lemma}[theorem]{Lemma}

\newtheorem{proposition}[theorem]{Proposition}

\newtheorem{remark}[theorem]{Remark}
\newtheorem{definition}[theorem]{Definition}
\newtheorem{example}[theorem]{Example}

\theoremstyle{definition} % This makes defns, etc., appear without italics

\newcommand{\N}{\mathbb{N}}
\newcommand{\Z}{\mathbb{Z}}

\newcommand{\C}{\mathbb{C}}

\newcommand{\scrF}{\mathscr{F}}
\newcommand{\scrH}{\mathscr{H}}
\newcommand{\scrL}{\mathscr{L}}
\newcommand{\scrS}{\mathscr{S}}
\newcommand{\scrT}{\mathscr{T}}
\newcommand{\scrU}{\mathscr{U}}
\newcommand{\scrY}{\mathscr{Y}}

\newcommand{\z}{\mathfrak{z}}
\newcommand{\w}{\mathfrak{w}}
\newcommand{\frakI}{\mathfrak{I}}
\newcommand{\frakZ}{\mathfrak{Z}}

\begin{document}

\title{Comparing Two Generalized Noncommutative Nevanlinna-Pick Theorems}

\author{Rachael M. Norton}
\address{Department of Mathematics, University of Iowa, Iowa City, IA 52242}
\email{rachael-norton@uiowa.edu}
%\urladdr{www.math.sc.edu/$\sim$second}

\begin{abstract}
We explore the relationship between two noncommutative generalizations of the classical Nevanlinna-Pick theorem: one proved by Constantinescu and Johnson in 2003 and the other proved by Muhly and Solel in 2004.  To make the comparison, we generalize Constantinescu and Johnson's theorem to the context of $W^*$-correspondences and Hardy algebras.  After formulating the so-called displacement equation in this context, we are able to follow Constantinescu and Johnson's line of reasoning in our proof.  Though our result is similar in appearance to Muhly and Solel's, closer inspection reveals differences.  Nevertheless, when the given data lie in the center of the dual correspondence, the theorems are essentially the same.          
\end{abstract}

\maketitle

\section{Introduction}

In this paper, we explore the relationship between two noncommutative generalizations of the famous Nevanlinna-Pick theorem: Constantinescu and Johnson's Theorem 3.4 in \cite{Constantinescu2003} and Muhly and Solel's Theorem 5.3 in \cite{Muhly2004a}. In Constantinescu and Johnson's theorem, the given data are $N$-tuples of operators on Hilbert space; the interpolating map is an upper triangular matrix with operator entries; and its existence depends on the positivity of the so-called Pick matrix.  Muhly and Solel, however, work in the setting of $W^*$-correspondences.  They interpolate points in the dual correspondence; the interpolating map belongs to the Hardy algebra $H^\infty(E)$ of the correspondence; and interpolation occurs when their Pick matrix is a completely positive map.  Furthermore, to prove their theorem Constantinescu and Johnson exploit the properties of the displacement equation while Muhly and Solel use the commutant lifting approach.  In order to compare the theorems, we generalize Constantinescu and Johnson's Theorem 3.4 to the context of $W^*$-correspondences and Hardy algebras.  Our proof follows the trajectory of the original proof of the theorem.  Once in this setting, we can consider the similarities and differences between the theorems. The point evaluation in Muhly and Solel's Theorem 5.3 is a homomorphism, while our point evaluation is not; it merely gives rise to an antihomomorphism on the Hardy algebra of the dual correspondence, $H^\infty(E^\sigma)$.  Furthermore, Muhly and Solel's  interpolating map belongs to $H^\infty(E)$ while ours belongs to $H^\infty(E^\sigma)$.  Nevertheless, in the case when the data lie in the center of the dual correspondence, $\frakZ(E^\sigma)$, there exists an interpolating map in $H^\infty(\frakZ(E^\sigma))$ if and only if there exists a map in $H^\infty(\frakZ(E))$ which interpolates the adjoints of the data.  Moreover, the interpolating maps are related by an isomorphism of the Hardy algebras.  Lastly, we give an equivalent condition for interpolation in terms of completely bounded maps.  

\section{Preliminaries}
\label{s: preliminaries}
Throughout this paper, $M$ will be a $W^*$-algebra.  We will think of $M$ as a $C^*$-algebra, without a preferred representation, that is also a dual space. Let $E$ denote a $W^*$-correspondence over $M$ in the sense of \cite[Definition 2.2]{Muhly2002a}.  That is, $E$ is a self-dual Hilbert $C^*$-bimodule over $M$.  The $M$-valued inner product on $E$ is full, and the left action of $M$ on $E$ is given by a faithful, normal $*$-homomorphism $\varphi:M \to \scrL(E)$, where $\scrL(E)$ denotes the $W^*$-algebra of adjointable operators on $E$.  For $k \in \N$, we form the tensor power of $E$, $E^{\otimes k}$, balanced over $M$.  $E^{\otimes k}$ is a $W^*$-correspondence over $M$ with the left action given by $\varphi_k: M \to \scrL(E^{\otimes k})$, where $\varphi_k(a)(\xi_1 \otimes \xi_2 \otimes \cdots \otimes \xi_k) = (\varphi(a)\xi_1) \otimes \xi_2 \otimes \cdots \otimes \xi_k.$ Let $E^{\otimes 0} = M$, viewed as a bimodule over itself, and define the \emph{Fock space} of $E$ to be the ultraweak direct sum $\scrF(E) = \bigoplus_{k=0}^\infty E^{\otimes k}$.  The Fock space of $E$ is also a $W^*$-correspondence over $M$.  We denote the left action of $M$ on $\scrF(E)$ by $\varphi_\infty$, defined by the formula $\varphi_\infty(a) = \text{diag}[a, \varphi(a), \varphi_2(a), \cdots ]$.  Define the \emph{(left) creation operators} $\{T_\xi \mid \xi \in E\}$ on $\scrF(E)$ by $T_\xi(\eta) = \xi \otimes \eta, \eta \in \scrF(E)$.  Matricially, 
\begin{equation} \label{eq: creationop}
T_\xi = \begin{bmatrix} 0 \\ T_\xi^{(1)} & 0 \\ & T_\xi^{(2)} & 0 \\ & & \ddots & \ddots \end{bmatrix},
\end{equation} 
where $T_\xi^{(k)}:E^{\otimes k-1} \to E^{\otimes k}$ is given by $T_\xi^{(k)}(\eta_1 \otimes \cdots \otimes \eta_{k-1}) = \xi \otimes \eta_1 \otimes \cdots \otimes \eta_{k-1}$.

The \emph{tensor algebra} over $E$, denoted $\scrT_+(E)$, is the norm-closed subalgebra of $\scrL(\scrF(E))$ generated by the left action operators $\{\varphi_\infty(a) \mid a \in M\}$ and the creation operators $\{T_\xi \mid \xi \in E\}$.  The \emph{Hardy algebra} of $E$ is the ultraweak closure of $\scrT_+(E)$ in $\scrL(\scrF(E))$ and is denoted by $H^\infty(E)$.

Let $\sigma : M \to B(H)$ be a faithful, normal representation of $M$ on a Hilbert space $H$.  Form $E \otimes_\sigma H$, the Hausdorff completion of the algebraic tensor product $E \otimes H$ in the positive semidefinite sesquilinear form defined by the formula $\langle \xi \otimes h, \eta \otimes k \rangle = \langle h, \sigma(\langle \xi, \eta \rangle) k \rangle$, for $\xi \otimes h, \eta \otimes k \in E \otimes_\sigma H$. Then $\sigma$ induces the representation $\sigma^E : \scrL(E) \to B(E \otimes_\sigma H)$ given by $\sigma^E(T) = T \otimes I_H$.  

Define the intertwining space $\frakI(\sigma, \sigma^E \circ \varphi) :=\{\eta \in B(H, E \otimes_\sigma H) \mid \eta \sigma(a) = \sigma^E \circ \varphi(a) \eta \quad \forall a \in M\}$.  For convenience, we will write $E^\sigma$ instead of $\frakI(\sigma, \sigma^E \circ \varphi)$, and we will refer to this space as the \emph{$\sigma$-dual} of $E$. $E^\sigma$ is a $W^*$-correspondence over $\sigma(M)'$, the commutant of $\sigma(M)$ in $B(H)$, under the following actions and inner product: for $a, b \in \sigma(M)'$ and $\eta, \xi \in E^\sigma$, $a \cdot \eta \cdot b := (I_E \otimes a) \eta b$ and $\langle \eta, \xi \rangle:= \eta^* \xi.$  We will denote the left action of $\sigma(M)'$ on $E^\sigma$ by $\varphi^\sigma$.  As above, form the tensor powers $(E^\sigma)^{\otimes k}, k \in \N,$ balanced over $\sigma(M)'$, and the Fock space $\scrF(E^\sigma)$.  The left action maps are denoted by $\varphi_k^\sigma$ and $\varphi_\infty^\sigma$, respectively.  Let $\iota :\sigma(M)' \to B(H)$ be the identity representation of $\sigma(M)'$ on $H$. Form $(E^\sigma)^{\otimes k} \otimes_\iota H$ and $\scrF(E^\sigma) \otimes_\iota H$. For $\eta \in E^\sigma$ and $k \in \N$, define $\eta^{(k)} \in \frakI(\sigma, \sigma^{E^{\otimes k}} \circ \varphi_k)$ by $\eta^{(k)} = (I_{E^{\otimes k-1}} \otimes \eta) (I_{E^{\otimes k-2}} \otimes \eta) \cdots (I_E \otimes \eta) \eta$.  Note that $\eta^{(k+1)} = (I_{E^{\otimes k}} \otimes \eta) \eta^{(k)}$. Then define the \emph{Cauchy Kernel} $C(\eta) \in B(H, \scrF(E) \otimes_\sigma H)$ by $C(\eta) = \begin{bmatrix} I_H & \eta & \eta^{(2)} & \eta^{(3)} & \cdots \end{bmatrix}^T$. For $\xi, \eta \in E^\sigma$, the inner product $\langle C(\xi), C(\eta) \rangle$ is given by the formula $\langle C(\xi), C(\eta) \rangle = C(\xi)^*C(\eta)$. 

In order to define a point evaluation for elements in $H^\infty(E^\sigma)$, we must first define a couple of maps.  As in \cite[Lemma 3.8]{Muhly2004a}, define $U:\scrF(E^\sigma) \otimes_\iota H \to \scrF(E) \otimes_\sigma H$ to be the Hilbert space direct sum $U = \bigoplus_{k=0}^\infty U_k$, where $U_k : (E^\sigma)^{\otimes k} \otimes_\iota H \to E^{\otimes k} \otimes_\sigma H$ is given by the formula $U_k(\eta_1 \otimes \eta_2 \otimes \cdots \otimes \eta_k \otimes h) = (I_{E^{\otimes k-1}} \otimes \eta_1)(I_{E^{\otimes k-2}} \otimes \eta_2) \cdots (I_E \otimes \eta_{k-1})\eta_kh.$ As in \cite[Theorem 3.9]{Muhly2004a}, define the ultraweakly continuous, completely isometric representation $\rho : H^\infty(E^\sigma) \to B(\scrF(E) \otimes_\sigma H)$ by 
\begin{equation} 
\label{eq: rho}
\rho(X) = U(X \otimes I_H)U^*, \quad X \in H^\infty(E^\sigma).
\end{equation}
Then for $X \in H^\infty(E^\sigma),$ we define a $\sigma(M)'$-valued point evaluation on $E^\sigma$ by the formula 
\begin{equation} \label{eq: pteval}
\hat{X}(\eta) = \langle \rho(X) C(0), C(\eta) \rangle, \quad \eta \in E^\sigma,
\end{equation}
where $C(0) = \begin{bmatrix} I_H & 0 & 0 & \cdots \end{bmatrix}^T$. Note that for $X, Y \in H^\infty(E^\sigma)$ and $\lambda \in \C$,  $\widehat{X+\lambda Y} = \hat{X} + \overline{\lambda}\hat{Y}$ since $\rho$ is linear.  While the point evaluation is not multiplicative, in Section \ref{s: remarksonpteval} we show how it gives rise to an antihomomorphism from $H^\infty(E^\sigma)$ into the completely bounded maps on $\sigma(M)'$.   

We are now ready to state our generalized Nevanlinna-Pick theorem.

\begin{theorem} 
\label{thm: NPtheorem} 
Let $E$ be a W$^*$-correspondence over a $W^*$-algebra $M$, with the left action of $M$ on $E$ given by a faithful, normal $*$-homomorphism $\varphi: M \to \scrL(E)$.  Let $\sigma$ be a faithful, normal representation of $M$ on a Hilbert space $H$.  Let $\z_1, \ldots, \z_N \in E^\sigma$ with $\|\z_i\| < 1$, for $ i=1, \ldots, N$, and 
$\Lambda_1, \ldots, \Lambda_N \in \sigma(M)'$.  There exists $X \in H^\infty(E^\sigma)$ with $\|X\| \leq 1$ such that 
\begin{eqnarray*}
\hat{X}(\z_i)=\Lambda_i, \quad i = 1, \ldots, N,
\end{eqnarray*}
if and only if the Pick matrix
\begin{eqnarray}\label{eq: pickmat}
A = \begin{bmatrix} \langle C(\z_i), C(\z_j) \rangle - \langle (I_{\scrF(E)} \otimes \Lambda_i) C(\z_i), (I_{\scrF(E)} \otimes \Lambda_j) C(\z_j) \rangle \end{bmatrix}_{i,j=1}^N
\end{eqnarray}
is positive semidefinite. \end{theorem}

Note that if we set $N = n, E = \C^N, M = \C,$ and $\sigma(a) = aI_H$ for all $a \in M$, then we recover Constantinescu and Johnson's Theorem 3.4 in \cite{Constantinescu2003}. In fact, we arrived at Theorem \ref{thm: NPtheorem} by generalizing \cite[Theorem 3.4]{Constantinescu2003}, and it lends itself most naturally to a comparison with Muhly and Solel's Theorem 5.3 in \cite{Muhly2004a}.  Nevertheless, a statement that avoids $E^\sigma$ may be preferable in some cases. 

We can state Theorem \ref{thm: NPtheorem} without reference to the $\sigma$-dual as follows.  Let $F$ be a $W^*$-correspondence over a $W^*$-algebra $P$.  Let $\tau : P \to B(H)$ be a faithful, normal representation of $P$ on a Hilbert space $H$.   For $\eta \in F$ and $X \in H^\infty(F)$, define 
\begin{equation*}
\tilde{X}(\eta) := \langle (X \otimes I_H) \tilde{C}(0), \tilde{C}(\eta)\rangle
\end{equation*}
to be the $P$-valued point evaluation of $X$ at $\eta$, where $\tilde{C}(\eta) = \begin{bmatrix} I_H & L^{(1)}_\eta & L^{(2)}_{\eta^{\otimes 2}} & \cdots \end{bmatrix}^T$ and $L^{(k)}_{\eta^{\otimes k}} : H \to F^{\otimes k} \otimes_\tau H$ is given by $L^{(k)}_{\eta^{\otimes k}}(h) = \eta^{\otimes k} \otimes h.$   By Theorem 3.6 in \cite{Muhly2004a}, there exists a $W^*$-correspondence $E$ over a $W^*$-algebra $M$ and a faithful, normal representation $\sigma:M \to B(H)$ such that $F=E^\sigma, P = \sigma(M)',$ and $\tau$ is the identity map.  Let $X \in H^\infty(F) =  H^\infty(E^\sigma),\eta \in F =  E^\sigma$, and $\Lambda \in P = \sigma(M)'$.  A simple calculation shows $\tilde{C}(\eta) = U^*C(\eta),$ and it immediately follows that $\tilde{X}(\eta) = \hat{X}(\eta).$ Furthermore, Lemma 3.8 in \cite{Muhly2004a} implies that $(I_{\scrF(E)} \otimes \Lambda)C(\eta) = U(\varphi_\infty^F(\Lambda) \otimes I_H)\tilde{C}(\eta)$.  Thus we arrive at the following theorem.

\begin{theorem}
Let $F$ be a $W^*$-correspondence over a $W^*$-algebra $P$.  Let $\z_1, \ldots, \z_N \in F$ with $\| \z_i \|<1, i=1, \ldots, N$, and $\Lambda_1, \ldots, \Lambda_N \in P$.  There exists $X \in H^\infty(F)$ with $\|X\| \leq 1$ such that 
\begin{equation*}
\tilde{X}(\z_i) = \Lambda_i, \quad  i=1, \ldots, N,
\end{equation*}
if and only if the Pick matrix 
\begin{eqnarray*}
\begin{bmatrix} \langle \tilde{C}(\z_i), \tilde{C}(\z_j) \rangle - \langle (\varphi_\infty^F(\Lambda_i) \otimes I_H)\tilde{C}(\z_i), (\varphi_\infty^F(\Lambda_j) \otimes I_H)\tilde{C}(\z_j) \rangle \end{bmatrix}_{i,j=1}^N
\end{eqnarray*}
is positive semidefinite.
\end{theorem}

For convenience, we state the nontangential version of Muhly and Solel's generalized Nevanlinna-Pick theorem \cite[Theorem 5.3]{Muhly2004a} alongside our results.  

\begin{theorem}[{\cite[Theorem 5.3]{Muhly2004a}}] \label{thm: MSNPtheorem}
Let $E$ be a $W^*$-correspondence over a $W^*$-algebra $M$, and let $\sigma$ be a faithful, normal representation of $M$ on $H$.  Given $\z_1, \ldots, \z_N \in E^\sigma$ with $\| \z_i\| <1, i=1, \ldots, N,$ and $\Lambda_1, \ldots, \Lambda_N \in B(H)$, there exists $Y \in H^\infty(E)$ with $\|Y\| \leq1$ such that 
\begin{equation*}
\hat{Y}(\z_i^*) = \Lambda_i, \quad i=1, \ldots, N,
\end{equation*}
if and only if the map from $M_N(\sigma(M)')$ to $M_N(B(H))$ defined by 
\begin{equation*}
[B_{ij}]_{i,j=1}^N \mapsto [\langle C(\z_i), (I_{\scrF(E)} \otimes B_{ij}) C(\z_j) \rangle - \langle C(\z_i) \Lambda_i^*, (I_{\scrF(E)} \otimes B_{ij}) C(\z_j) \Lambda_j^*\rangle ]_{i,j=1}^N
\end{equation*}
is completely positive, where the point evaluation of $Y$ at $\z_i^*$ is given by the formula $\hat{Y}(\z_i^*) = \langle (Y \otimes I_H) C(0), C(\z_i) \rangle^*$.  
\end{theorem}

In Section \ref{s: remarksonpteval} we compare Theorems \ref{thm: NPtheorem} and \ref{thm: MSNPtheorem}, and we give a condition for when the two theorems are equivalent.  For now we focus on one difference between the two theorems:  If the map in Theorem \ref{thm: MSNPtheorem} is completely positive, then by setting $B_{ij} = I_H$ for all $i,j=1, \ldots, N$, we see that the matrix

\begin{equation*}
[\langle C(\z_i), C(\z_j) \rangle - \langle C(\z_i) \Lambda_i^*, C(\z_j) \Lambda_j^* \rangle ]_{i,j=1}^N
\end{equation*}
is positive.  Observe that this matrix is almost identical to the Pick matrix $A$ in equation \eqref{eq: pickmat}.  Nevertheless, its positivity is a neccesary but not sufficient condition for interpolation in Theorem \ref{thm: MSNPtheorem}, while the positivity of \eqref{eq: pickmat} is a necessary and sufficient condition for interpolation in Theorem \ref{thm: NPtheorem}. The following simple example, brought to our attention by the referee, illustrates this point.  

\begin{example} \label{ex: example}
Let $Z$ and $\Lambda$ be $2 \times 2$ matrices given by
\begin{equation*}
Z = \begin{bmatrix} 0 & r \\0 & 0 \end{bmatrix}, 0 < r < 1, \text{ and } \Lambda = \begin{bmatrix} \epsilon & 0 \\ 0 & 0 \end{bmatrix}, 0 < \epsilon \leq 1.
\end{equation*}
Now consider two problems:
\begin{enumerate}
\item Find $F(z) = \sum_{n=0}^\infty A_n z^n$ in the unit ball of $H^\infty \otimes \C^{2 \times 2}$ such that
\begin{equation*}
F(Z) := \sum_{n=0}^\infty A_n Z^n = \Lambda,
\end{equation*}
where $A_n Z^n$ is given by multiplication of $2 \times 2$ matrices.

\item Find $f(z) = \sum_{n=0}^\infty a_n z^n$ in the unit ball of $H^\infty$ such that 
\begin{equation*}
f(Z) := \sum_{n=0}^\infty a_n Z^n = \Lambda,
\end{equation*}
where $a_n Z^n$ is given by scalar multiplication of a matrix.
\end{enumerate}

It can be shown that the first problem is a specific case of Constantinescu and Johnson's Theorem 3.4 in \cite{Constantinescu2003}.  Consequently, interpolation occurs if and only if $\sum_{n=0}^\infty Z^{*n}(I-\Lambda^* \Lambda)Z^n \geq 0$, which is the case since $\| \Lambda \| \leq 1$. On the other hand, the second problem is a specific case of Theorem \ref{thm: MSNPtheorem}.  One can easily check that it has no solution. One can also show that the associated map in Theorem \ref{thm: MSNPtheorem}
\begin{equation} \label{eq: map}
B \mapsto \sum_{n=0}^\infty Z^{*n} B Z^n - \Lambda^* Z^{*n} B Z^n \Lambda
\end{equation}
is not completely positive by applying Choi's criterion \cite[Theorem 2]{Choi1975}. 
\end{example}

 Thus in Example \ref{ex: example}, Constantinescu and Johnson have interpolation but Muhly and Solel do not, despite the fact that we get a positive matrix when we evaluate equation \eqref{eq: map} at $B=I_2$.

\section{Interpolating Maps}
Since the proof of Theorem \ref{thm: NPtheorem} is adapted from Constantinescu and Johnson's proof of Theorem 3.4 in \cite{Constantinescu2003}, it will be useful to restate some of their definitions in the context of $W^*$-correspondences.  In Theorem 3.4 in \cite{Constantinescu2003}, the interpolating map is a contraction that belongs to an algebra of upper triangular operators (see equation (3.1) in \cite{Constantinescu2003}).  We take this opportunity to define this algebra in our setting and examine its relationship to $H^\infty(E^\sigma)$. 

Given a $W^*$-correspondence $E$ over a $W^*$-algebra $M$ and a faithful, normal representation $\sigma$ of $M$ on a Hilbert space $H$, define $\scrU_\scrT(E, H,\sigma)$ to be the algebra of upper triangular operators $T = \begin{bmatrix} T_{ij} \end{bmatrix}_{i,j=0}^\infty \in B(\scrF(E) \otimes_\sigma H)$ such that
$T_{0j} \in \mathfrak{I}(\sigma^{E^{\otimes j}} \circ \varphi_j, \sigma),$ for $j \geq 0$, and $T_{ij} = I_E \otimes T_{i-1,j-1},$ for $1 \leq i \leq j$. That is, $T$ is a bounded, linear operator on $\scrF(E) \otimes_\sigma H$ of the form 
\begin{equation} \label{eq: Schurclass}
T = \begin{bmatrix} T_{00} & T_{01} & T_{02} & T_{03} & \cdots \\
0 & I_E \otimes T_{00} & I_E \otimes T_{01} & I_E \otimes T_{02} & \cdots \\
0 & 0 & I_{E^{\otimes 2}} \otimes T_{00} & I_{E^{\otimes 2}} \otimes T_{01} & 
\cdots \\
0 & 0 & 0 & \vdots & \ddots
\end{bmatrix},
\end{equation}
and $T_{0j} (\sigma^{E^{\otimes j}} \circ \varphi_j(a)) = \sigma(a) T_{0j}$ for all $a \in M$ and $j \geq 0.$  The collection of contractions in $\scrU_\scrT(E, H,\sigma)$ is called the \emph{Schur class} and is denoted by $\scrS(E, H,\sigma)$.  

The connection between $\scrU_\scrT(E,H,\sigma)$ and $H^\infty(E^\sigma)$ is made precise by the following lemma.

\begin{lemma}
\label{thm: equalsets}
Define $\rho:H^\infty(E^\sigma) \to B(\scrF(E) \otimes_\sigma H)$ as in equation \eqref{eq: rho}.  Then $\scrU_\scrT(E,H,\sigma)^* = \rho(H^\infty(E^\sigma))$.   
\end{lemma}

\begin{proof} 
In \cite[Theorem 3.9]{Muhly2004a}, Muhly and Solel showed that $\rho(H^\infty(E^\sigma)) = \sigma^{\scrF(E)}(H^\infty(E))'.$  To show that $\scrU_\scrT(E,H,\sigma)^* \subseteq \rho(H^\infty(E^\sigma))$, it suffices to show that every element of $\scrU_\scrT(E,H,\sigma)^*$ commutes with the generators of $ \sigma^{\scrF(E)}(H^\infty(E))$.  That is, we must show that every element of $\scrU_\scrT(E,H,\sigma)^*$ commutes with $\sigma^{\scrF(E)}(\varphi_\infty(a)), a \in M,$ and with $\sigma^{\scrF(E)}(T_\xi), \xi \in E$. 

For $a \in M$ and $T \in \scrU_\scrT(E,H,\sigma)$,
\begin{multline*}
T^* \circ \sigma^{\scrF(E)}(\varphi_\infty(a)) = 
\begin{bmatrix} T_{00}^* & 0 & 0 & \cdots \\ T_{01}^* & I_E \otimes T_{00}^* & 0 & \cdots \\
T_{02}^* & I_E \otimes T_{01}^* & I_{E^{\otimes 2}} \otimes T_{00}^* & \cdots \\
\vdots & \vdots & \vdots & \ddots 
\end{bmatrix} 
\begin{bmatrix} a \otimes I_H \\ & \varphi(a) \otimes I_H \\ & & \varphi_2(a) \otimes I_H \\ & & & \ddots \end{bmatrix} \\
= \begin{bmatrix} T_{00}^* \sigma(a)  & 0 & 0 & \cdots \\ T_{01}^* \sigma(a) & \varphi(a) \otimes T_{00}^* & 0 & \cdots \\ T_{02}^*\sigma(a) & \varphi(a) \otimes T_{01}^* & \varphi_2(a) \otimes T_{00}^* & \cdots \\ \vdots & \vdots & \vdots & \ddots 
\end{bmatrix}
= \sigma^{\scrF(E)}(\varphi_\infty(a)) \circ T^*
\end{multline*}
since $T_{0j} \in \mathfrak{I}(\sigma^{E^{\otimes j}} \circ \varphi_j, \sigma)$ and $a \otimes I_H = \sigma(a)$.

For $\xi \in E$ and $T \in \scrU_\scrT(E,H,\sigma)$,
\begin{multline*}
T^* \circ \sigma^{\scrF(E)}(T_\xi) = 
\begin{bmatrix} T_{00}^* & 0 & 0 & \cdots \\ T_{01}^* & I_E \otimes T_{00}^* & 0 & \cdots \\
T_{02}^* & I_E \otimes T_{01}^* & I_{E^{\otimes 2}} \otimes T_{00}^* & \cdots \\
\vdots & \vdots & \vdots & \ddots 
\end{bmatrix} 
\begin{bmatrix} 0 & 0 & 0 & \cdots \\T_\xi^{(1)} \otimes I_H & 0 & 0 & \cdots \\ 0 & T_\xi^{(2)} \otimes I_H & 0 & \cdots \\ \vdots & \vdots & \vdots & \ddots
\end{bmatrix} \\
= \begin{bmatrix} 0 & 0 & 0 & \cdots \\ (I_E \otimes T_{00}^*)(T_\xi^{(1)} \otimes I_H) & 0 & 0 & \cdots \\ (I_E \otimes T_{01}^*)(T_\xi^{(1)} \otimes I_H) & (I_{E^{\otimes 2}} \otimes T_{00}^*)(T_\xi^{(2)} \otimes I_H) & 0 & \cdots \\
(I_E \otimes T_{02}^*)(T_\xi^{(1)} \otimes I_H) & (I_{E^{\otimes 2}} \otimes T_{01}^*)(T_\xi^{(2)} \otimes I_H) & (I_{E^{\otimes 3}} \otimes T_{00}^*)(T_\xi^{(3)} \otimes I_H) & \cdots \\
\vdots & \vdots & \vdots & \ddots \end{bmatrix} \\
= \sigma^{\scrF(E)}(T_\xi) \circ T^*
\end{multline*}
because 
\begin{multline*}
(I_{E^{\otimes k}} \otimes T_{0,i-k}^*)(T_\xi^{(k)} \otimes I_H)(\eta_1 \otimes \cdots \otimes \eta_{k-1} \otimes h) = (I_{E^{\otimes k}} \otimes T_{0,i-k}^*)(\xi \otimes \eta_1 \otimes \cdots \otimes \eta_{k-1} \otimes h) \\
= \xi \otimes \eta_1 \otimes \cdots \otimes \eta_{k-1} \otimes T_{0,i-k}^*(h) 
= (T_\xi^{(i)} \otimes I_H)(I_{E^{\otimes k-1}} \otimes T_{0,i-k}^*)(\eta_1 \otimes \cdots \eta_{k-1} \otimes h).
\end{multline*}
For the other inclusion, note that it is a consequence of \cite[Theorem 3.9]{Muhly2004a} that $\rho(\varphi_\infty^\sigma(a)) = I_{\scrF(E)} \otimes a, a \in \sigma(M)',$ and 
\begin{equation*}
\rho(T_\eta) = \begin{bmatrix} 0 \\ \eta & 0 \\ & I_E \otimes \eta & 0 \\ & & I_{E^{\otimes 2}} \otimes \eta & \ddots \\ & & & \ddots \end{bmatrix}, \quad \eta \in E^\sigma.
\end{equation*}
Now it is easy to see that $\rho(\varphi_\infty^\sigma(a))$ and $\rho(T_\eta)$ are elements of $\scrU_\scrT(E,H,\sigma)^*$.
\end{proof}

Since $\scrU_\scrT(E,H,\sigma)^* = \rho(H^\infty(E^\sigma))$, we define the point evaluation of an element in $\scrU_\scrT(E,H,\sigma)$ at a point in $E^\sigma$ to agree with equation \eqref{eq: pteval}. That is, for $T \in \scrU_\scrT(E,H,\sigma)$ and $\eta \in E^\sigma$, $T(\eta) \in \sigma(M)'$ is given by the formula
\begin{equation*}
T(\eta) = \langle C(0), T C(\eta) \rangle.
\end{equation*}

The following result provides more information about the point evaluation and will be useful in the proof of Theorem \ref{thm: NPtheorem}.
\begin{lemma} \label{thm: pteval} If $T \in \scrU_\scrT(E, H,\sigma)$ and $\eta \in E^\sigma$, then $TC(\eta) = (I_{\scrF(E)} \otimes T(\eta))C(\eta)$. \end{lemma}

\begin{proof}
Expand the left hand side:
\begin{equation*}
TC(\eta) = \begin{bmatrix} T_{00} & T_{01} & T_{02} & \cdots \\ 
0 & I_E \otimes T_{00} & I_E \otimes T_{01} & \cdots \\
0 & 0 & I_{E^{\otimes 2}} \otimes T_{00} & \cdots \\
\vdots & \vdots & \vdots & \ddots
\end{bmatrix} \begin{bmatrix} I_H \\ \eta \\ \eta^{(2)} \\ \vdots \end{bmatrix} 
= \begin{bmatrix} \sum_{r=0}^\infty T_{0r}\eta^{(r)} \\
 \sum_{r=0}^\infty (I_E \otimes T_{0r})\eta^{(r+1)} \\
 \sum_{r=0}^\infty (I_{E^{\otimes 2}} \otimes T_{0r})\eta^{(r+2)} \\
\vdots
\end{bmatrix}.
\end{equation*}
Expand the right hand side:
\begin{multline*}
(I_{\scrF(E)} \otimes T(\eta))C(\eta) = \begin{bmatrix} T(\eta) \\ & I_E \otimes T(\eta) \\ & & I_{E^{\otimes 2}} \otimes T(\eta) \\ & & & \ddots \end{bmatrix} 
\begin{bmatrix} I_H \\ \eta \\ \eta^{(2)} \\ \vdots \end{bmatrix} 
= \begin{bmatrix} T(\eta) \\ (I_E \otimes T(\eta))\eta \\ (I_{E^{\otimes 2}} \otimes T(\eta))\eta^{(2)} \\ \vdots \end{bmatrix}.
\end{multline*}
The $k^\text{th}$ entry of the right hand side is
\begin{multline*} 
(I_{E^{\otimes k}} \otimes T(\eta))\eta^{(k)} = (I_{E^{\otimes k}} \otimes C(0)^*TC(\eta))\eta^{(k)} 
= (I_{E^{\otimes k}} \otimes\sum_{r=0}^\infty T_{0r}\eta^{(r)})\eta^{(k)} \\
= \sum_{r=0}^\infty \big(I_{E^{\otimes k}} \otimes (T_{0r}(I_{E^{\otimes r-1}} \otimes \eta) \ldots (I_E \otimes \eta)\eta)\big)\eta^{(k)} \\
= \sum_{r=0}^\infty (I_{E^{\otimes k}} \otimes T_{0r})(I_{E^{\otimes k}} \otimes I_{E^{\otimes r-1}} \otimes \eta) \ldots (I_{E^{\otimes k}} \otimes I_E \otimes \eta)(I_{E^{\otimes k}} \otimes \eta)\eta^{(k)}
= \sum_{r=0}^\infty (I_{E^{\otimes k}} \otimes T_{0r})\eta^{(r+k)}
\end{multline*}
which agrees with the $k^\text{th}$ entry of the left hand side.
\end{proof}

\section{Displacement Equation}
\label{s: diseq}
The displacement equation was originally defined by Kailath, Kung, and Morf in \cite{Kailath1979}, and it was used to measure the extent to which a matrix was Toeplitz (see also \cite{Kailath1995}). We are interested in a displacement equation of the form 
\begin{equation*}
A-\theta(A) = B,
\end{equation*}
where $\theta$ is a completely positive, contractive map.  In this case, one can solve for the unique solution $A$ by computing the resolvent, $A = (I- \theta)^{-1}(B)$.

In order to apply the displacement theory to our context, we first fix $\z_1, \z_2, \ldots, \z_N \in E^\sigma$ with $\|\z_i\| <1, i =1, \ldots, N,$ and $\Lambda_1, \ldots, \Lambda_N \in \sigma(M)'$, and we form the matrices 
\begin{equation} \label{eq: interdata}
U = \begin{bmatrix} I_H \\ \vdots \\ I_H \end{bmatrix}, V = \begin{bmatrix} \Lambda_1^* \\ \vdots \\ \Lambda_N^* \end{bmatrix}, \text{ and } \z = \begin{bmatrix} \z_1 \\ & \ddots \\ & & \z_N \end{bmatrix}.
\end{equation}
For the remainder of this section, we reserve this notation for these specific matrices.  We emphasize that $U$ defined in equation \eqref{eq: interdata} is in accord with \cite{Constantinescu2003} and should not be confused with the isomorphism in equation \eqref{eq: rho}. 

 Let $H^{(N)}$ denote $H \otimes \C^N$, and let $\sigma^{(N)} : M \to B(H^{(N)})$ be the representation of $M$ on $H^{(N)}$ given by the $N \times N$ diagonal matrix
\begin{equation*}
\sigma^{(N)}(a) = \begin{bmatrix} \sigma(a) \\ & \ddots \\ & & \sigma(a) \end{bmatrix}, \quad a \in M.
\end{equation*}
As in Section \ref{s: preliminaries}, define the intertwining space $\frakI(\sigma^{(N)}, (\sigma^{(N)})^E \circ \varphi) := \{\eta \in B( H^{(N)}, E \otimes_{\sigma^{(N)}} H^{(N)}) \mid \eta \sigma^{(N)}(a) = (\sigma^{(N)})^E \circ \varphi(a) \eta \quad \forall a \in M\}$. Also define $\eta^{(k)}$ and the Cauchy Kernel $C(\eta)$ for $\eta \in \frakI(\sigma^{(N)}, (\sigma^{(N)})^E \circ \varphi)$. Observe that $\z$ from equation \eqref{eq: interdata} belongs to $\frakI(\sigma^{(N)}, (\sigma^{(N)})^E \circ \varphi)$ and $\|\z\| <1$.  Consider the displacement equation 
\begin{equation} \label{eq: diseq}
A - \z^*(I_E \otimes A)\z = UU^* - VV^*.
\end{equation} 
Equation \eqref{eq: diseq} admits a unique solution $A \in \sigma^{(N)}(M)'$.  To see this, define $\theta_\z: \sigma^{(N)}(M)' \to \sigma^{(N)}(M)'$ by $\theta_\z(B) = \z^*(I_E \otimes B)\z$. 
Then for $k \in \N$, $\theta_\z^k(B) = (\z^{(k)})^*(I_{E^{\otimes 
k}} \otimes B)\z^{(k)}$. Since $\|\theta_\z\| \leq \|\z \|^2 <1,$ $(I_{B(H^{(N)})}-\theta_\z)^{-1}$ is a completely bounded map on $\sigma^{(N)}(M)'$. Consequently, we can solve equation \eqref{eq: diseq} for $A$:
\begin{multline*}
A = (I_{B(H^{(N)})} - \theta_\z)^{-1}(UU^* - VV^*) = \sum_{k=0}^\infty \theta_\z^k(UU^*-VV^*) 
= \sum_{k=0}^\infty (\z^{(k)})^*(I_{E^{\otimes k}} \otimes (UU^*-VV^*))\z^{(k)} \\
= \sum_{k=0}^\infty  (\z^{(k)})^*(I_{E^{\otimes k}} \otimes 
UU^*)\z^{(k)} - \sum_{k=0}^\infty  (\z^{(k)})^*(I_{E^{\otimes k}} 
\otimes VV^*)\z^{(k)} \\
= C(\z)^*(I_{\scrF(E)} \otimes UU^*)C(\z) - C(\z)^*(I_{\scrF(E)} \otimes VV^*) C(\z) \\
= \begin{bmatrix} \langle C(\z_i), C(\z_j) \rangle - \langle (I_{\scrF(E)} \otimes \Lambda_i) C(\z_i), (I_{\scrF(E)} \otimes \Lambda_j) C(\z_j) \rangle \end{bmatrix}_{i,j=1}^N,
\end{multline*}
which is the Pick matrix from equation \eqref{eq: pickmat}.

In the proof of Theorem \ref{thm: NPtheorem}, it will be convenient to write $A$ in terms of different notation.  Thus define two maps $U_\infty^*$ and $V_\infty^*$ both from $\scrF(E) \otimes H$ to $H^{(N)}$ by
\begin{eqnarray*}
U_\infty^* &=& \begin{bmatrix} U & \z^*(I_E \otimes U) & 
(\z^{(2)})^*(I_{E^{\otimes 2}} \otimes U) & \ldots \end{bmatrix} \\
V_\infty^* &=& \begin{bmatrix} V & \z^*(I_E \otimes V) & 
(\z^{(2)})^*(I_{E^{\otimes 2}} \otimes V) & \ldots \end{bmatrix}.
\end{eqnarray*}
Then $A = U_\infty^* U_\infty - V_\infty^* V_\infty$.  Note that we may rewrite $U_\infty$ and $V_\infty$ in terms of the Cauchy kernels as follows:
\begin{equation*} 
U_\infty = \begin{bmatrix} C(\z_1) & \cdots & C(\z_N) \end{bmatrix}
\end{equation*}
and  
\begin{equation*} 
V_\infty = \begin{bmatrix} (I_{\scrF(E)} \otimes \Lambda_1)C(\z_1) & \cdots & (I_{\scrF(E)} \otimes \Lambda_N)C(\z_N) \end{bmatrix}.
\end{equation*}

These observations will be useful later, so we summarize them in the following remark.

\begin{remark} \label{rk: pickmatrix} The Pick matrix \eqref{eq: pickmat} is the unique solution to the displacement equation \eqref{eq: diseq}, and it may be written in the form $A = U_\infty^* U_\infty - V_\infty^* V_\infty$, where $U_\infty = \begin{bmatrix} C(\z_1) & \cdots & C(\z_N) \end{bmatrix}$ and $V_\infty = \begin{bmatrix} (I_{\scrF(E)} \otimes \Lambda_1)C(\z_1) & \cdots & (I_{\scrF(E)} \otimes \Lambda_N)C(\z_N) \end{bmatrix}.$ \end{remark}

The following lemma is the crux of the proof of Theorem \ref{thm: NPtheorem}.  It relates the positivity of the Pick matrix to the existence of a special element in the Schur class, $\scrS(E, H,\sigma)$.  Recall that $\scrS(E, H,\sigma)$ is defined to be the collection of contractive upper triangular operators $T=[T_{ij}]_{i,j=0}^\infty$ of the form \eqref{eq: Schurclass} and satisfying $T_{0j} (\sigma^{E^{\otimes j}} \circ \varphi_j(a)) = \sigma(a) T_{0j}$ for all $a \in M$ and $j \geq 0.$ 

\begin{lemma} \label{thm: biglemma} The solution to the displacement equation \eqref{eq: diseq} is positive semidefinite if and only if there exists $T \in \scrS(E, H,\sigma)$ such that $TU_\infty = V_\infty$. \end{lemma}

In order to prove Lemma \ref{thm: biglemma}, we will need the following two propositions.  Proposition \ref{thm: transmap} is a result about transfer maps of time varying systems.  The \emph{state-space model} of a discrete time varying linear system is defined by an equation of the form 
\begin{equation} \label{eq: timevarying}
\begin{bmatrix} x(t) \\ y(t) \end{bmatrix} = \begin{bmatrix} A(t) & B(t) \\ C(t) & D(t) \end{bmatrix} \begin{bmatrix} x(t+1) \\ u(t) \end{bmatrix}, \quad t \in \Z,
\end{equation}
where $\{\scrU(t)\}_{t \in \Z}, \{\scrY(t)\}_{t \in \Z}$, and $\{\scrH(t)\}_{t \in \Z}$ are given families of Hilbert spaces called the input, output, and state spaces, respectively, and $u(t) \in \scrU(t), y(t) \in \scrY(t),$ and $x(t) \in \scrH(t)$ for all $t \in \Z$.  The operators $A(t) \in B(\scrH(t+1), \scrH(t)), B(t) \in B(\scrU(t), \scrH(t)), C(t) \in B(\scrH(t+1), \scrY(t))$, and $D(t) \in B(\scrU(t), \scrY(t))$ are also given.  The system \eqref{eq: timevarying} is said to be \emph{contractive} if $\left\| \begin{bmatrix} A(t) & B(t) \\ C(t) & D(t) \end{bmatrix} \right\| \leq 1$ for all $t \in \Z$.  The operators $A(t), B(t), C(t),$ and $D(t)$ uniquely determine the so-called \emph{transfer map} of the system, an operator $T: \bigoplus_{t \in \Z} \scrU(t) \to \bigoplus_{t \in \Z} \scrY(t)$ that satisfies $T(u(t))_{t \in \Z} = (y(t))_{t \in \Z}$.  For more on time varying linear systems, see \cite[Section 2.3]{Constantinescu1996}. The following result is derived from the proof of Lemma 3.1 in \cite[Section 2.3]{Constantinescu1996}. 

\begin{proposition} \label{thm: transmap} The transfer map of a contractive time varying linear system is a contraction.  \end{proposition}

\begin{proof}
Fix $t_0 \in \Z$.  Suppose $x(t_0)=0$, and let $\{y(t)\}_{t<t_0}$ be the output generated from the input $\{u(t)\}_{t<t_0}$  by the contractive time varying linear system 
\begin{eqnarray*}
\begin{bmatrix} x(t) \\ y(t) \end{bmatrix} &=& \begin{bmatrix} A(t) & B(t) \\ C(t) & D(t) \end{bmatrix} \begin{bmatrix} x(t+1) \\ u(t) \end{bmatrix}.
\end{eqnarray*}
Since $\left\| \begin{bmatrix} A(t) & B(t) \\ C(t) & D(t) \end{bmatrix} \right\| \leq 1,$ we have
\begin{eqnarray*}
\|x(t)\|^2 + \|y(t)\|^2 \leq \|x(t+1)\|^2 + \|u(t)\|^2.
\end{eqnarray*}
By induction, 
\begin{eqnarray*}
\|x(t)\|^2 \leq \sum_{k=t}^{t_0-1} \|u(k)\|^2 - \sum_{k=t}^{t_0-1} \|y(k)\|^2, \quad t<t_0.
\end{eqnarray*}
In particular, 
\begin{eqnarray*}
\sum_{k=t}^{t_0-1} \|y(k)\|^2 \leq \sum_{k=t}^{t_0-1} \|u(k)\|^2.
\end{eqnarray*}
Since this holds for arbitrary $t_0 \in \Z$ and arbitrary $t<t_0$, it follows that the transfer map $T$ is a contraction. 
\end{proof}

Proposition \ref{thm: Omegalemma} will imply that if the solution $A$ to the displacement equation \eqref{eq: diseq} is positive semidefinite, then the entries of the operator $T$ from Lemma \ref{thm: biglemma} satisfy the necessary intertwining relations in order for $T$ to belong to $\scrS(E, H, \sigma)$.  First note that if $A \geq 0$, then there exists $L \in \sigma^{(N)}(M)'$ such that 
$A = LL^*$.  
When we rewrite the displacement equation in terms of $L$, we get
\begin{eqnarray*} 
LL^* - \z^*(I_E \otimes LL^*)\z = UU^* - VV^*.
\end{eqnarray*} 
Define $\hat{A} := \begin{bmatrix} L^* \\ V^* \end{bmatrix}$ and $\hat{B} := 
\begin{bmatrix} (I_E \otimes L^*)\z \\ U^* \end{bmatrix}$.
\begin{proposition} \label{thm: Omegalemma}
If the solution to the displacement equation \eqref{eq: diseq} is positive semidefinite, then there exists a unique partial isometry $\Omega: (E \otimes H^{(N)}) \oplus H \to H^{(N+1)}$ such that $\hat{A} = \Omega \hat{B}$ and $ker(\Omega)^\perp \subseteq \overline{Range(\hat{B})}$.  Moreover, for all $a \in M$, 
\begin{equation} \label{eq: relation}
\sigma^{(N+1)}(a)\Omega =  \Omega \begin{bmatrix} (\sigma^{(N)})^E \circ \varphi(a) & 0 \\ 0 & \sigma(a) \end{bmatrix}.
\end{equation}
\end{proposition}

\begin{proof}
If the solution $A$ to the displacement equation is positive semidefinite, then we can rewrite the displacement equation as follows:
\begin{equation*}
\hat{A}^* \hat{A} = \hat{B}^* \hat{B},
\end{equation*}
where $\hat{A}$ and $\hat{B}$ are defined above.
By Douglas's Lemma  \cite[Theorem 1]{Douglas1966a}, there exists a unique partial isometry $\Omega: (E \otimes H^{(N)}) \oplus H \to H^{(N+1)}$ such that $\hat{A} = \Omega \hat{B}$ and $ker(\Omega)^\perp \subseteq \overline{Range(\hat{B})}$. Lastly, we must show that equation \eqref{eq: relation} holds. Recall that since $M$ is a $W^*$-algebra, it is generated by its unitaries.  Thus it suffices to prove equation \eqref{eq: relation} for all unitary elements of $M$.  Let $u \in M$ be unitary, and define the partial isometry $\hat{\Omega} = \sigma^{(N+1)}(u^*) \Omega \begin{bmatrix} (\sigma^{(N)})^E \circ \varphi(u) & 0 \\ 0 & \sigma(u) \end{bmatrix}$.  We will show $\hat{\Omega} = \Omega$.   

Note that the intertwining relations satisfied by the entries of $\hat{A}$ and $\hat{B}$ imply
\begin{equation} \label{eq: relations1}
\sigma^{(N+1)}(a)\hat{A} = \hat{A} \sigma^{(N)}(a) \text{ and }
\begin{bmatrix} (\sigma^{(N)})^E \circ \varphi(a) & 0 \\ 0 & \sigma(a) \end{bmatrix} \hat{B} = \hat{B} \sigma^{(N)}(a)
\end{equation}
for all $a \in M$.  Then $\hat{A} = \hat{\Omega} \hat{B}$ since $\hat{A} = \Omega \hat{B}$ and equation \eqref{eq: relations1} holds.  By the uniqueness of $\Omega$, it remains to show that $ker(\hat{\Omega})^\perp \subseteq \overline{Range(\hat{B})}$. That is, we must show $P \leq Q$, where $P$ is projection onto $ker(\hat{\Omega})^\perp$ and $Q$ is projection onto $\overline{Range(\hat{B})}$.  

Observe that $Q$ commutes with $\begin{bmatrix} (\sigma^{(N)})^E \circ \varphi(a) & 0 \\ 0 & \sigma(a) \end{bmatrix}$ for all $a \in M$ since \eqref{eq: relations1} holds.  Thus \begin{multline*}
P = \hat{\Omega}^* \hat{\Omega} = \begin{bmatrix} (\sigma^{(N)})^E \circ \varphi(u^*) & 0 \\ 0 & \sigma(u^*) \end{bmatrix} \Omega^* \Omega \begin{bmatrix} (\sigma^{(N)})^E \circ \varphi(u) & 0 \\ 0 & \sigma(u) \end{bmatrix} \\
\leq \begin{bmatrix} (\sigma^{(N)})^E \circ \varphi(u^*) & 0 \\ 0 & \sigma(u^*) \end{bmatrix} Q \begin{bmatrix} (\sigma^{(N)})^E \circ \varphi(u) & 0 \\ 0 & \sigma(u) \end{bmatrix} = Q,
\end{multline*}
where the inequality follows from the fact that $ker(\Omega)^\perp \subseteq \overline{Range(\hat{B})}$.

\end{proof}

\begin{proof}[Proof of Lemma \ref{thm: biglemma}]
Suppose the solution $A$ to the displacement equation \eqref{eq: diseq} is positive semidefinite, and let $\Omega$ be as in Proposition \ref{thm: Omegalemma}.  Then we may write $\Omega= \begin{bmatrix} X & Z \\ Y & W \end{bmatrix},$ for some $X \in B(E \otimes H^{(N)}, H^{(N)}), Z \in B(H, H^{(N)}), Y \in B(E \otimes H^{(N)}, H),$ and $W \in B(H)$.  The following intertwining relations are a consequence of equation \eqref{eq: relation}:
\begin{equation} \label{eq: interrelations}
X \in \mathfrak{I}((\sigma^{(N)})^E \circ \varphi,\sigma^{(N)}), \quad Z \in \mathfrak{I}(\sigma, \sigma^{(N)}), \quad Y \in \mathfrak{I}((\sigma^{(N)})^E \circ \varphi, \sigma), \quad W \in \sigma(M)'.
\end{equation}

Writing $\hat{A} = \Omega \hat{B}$ in terms of the entries of $\hat{A}, \Omega,$ and $\hat{B}$, we get the system of equations
\begin{equation} \label{eq: syseq}
\begin{aligned}
L^* &= X(I_E \otimes L^*)\z + ZU^* \\
V^* &= Y(I_E \otimes L^*)\z + WU^*.
\end{aligned}
\end{equation}
After substituting the first equation into the second $K$ times, we get
\begin{multline} \label{eq: subeq}
V^* = WU^* + \sum_{k=0}^{K-1} Y(I_E \otimes ((X^*)^{(k)})^*)(I_{E^{\otimes k+1}} \otimes ZU^*)\z^{(k+1)} \\
+Y(I_E \otimes ((X^*)^{(K)})^*)(I_{E^{\otimes K+1}} \otimes L^*)\z^{(K+1)}.
\end{multline} 
We can bound the last term in equation \eqref{eq: subeq} by
\begin{multline*}
\| Y(I_E \otimes ((X^*)^{(K)})^*)(I_{E^{\otimes K+1}} \otimes L^*)\z^{(K+1)}\|
\leq \|Y\| \|((X^*)^{(K)})^*\| \|L^*\| \|\z^{(K+1)}\| \\
\leq \|Y\| \|X\|^K \|L^*\| \|\z\|^{K+1}.
\end{multline*}
Since $\| \z\| < 1$ and $\|X\| \leq 1$, the last term goes to $0$ as $K$ goes to infinity, which shows
\begin{eqnarray*}
V^* = WU^* + \sum_{k=0}^{\infty} Y(I_E \otimes ((X^*)^{(k)})^*)(I_{E^{\otimes k+1}} \otimes ZU^*)\z^{(k+1)}.
\end{eqnarray*}
Form the infinite upper triangular matrix $T =[T_{ij}]_{i,j=0}^\infty$ defined as follows:
\begin{displaymath}
T_{ij} = \left\{ 
\begin{array} {lr}
0 & j<i \\
I_{E^{\otimes i}} \otimes W & j=i \\
I_{E^{\otimes i}} \otimes Y(I_E \otimes Z) & j=i+1 \\
I_{E^{\otimes i}} \otimes Y(I_E \otimes ((X^*)^{(j-i-1)})^*)(I_{E^{\otimes j-i}} \otimes Z) & j > i+1
\end{array}
\right.
\end{displaymath}
That is, 
\begin{equation*}
T = \begin{bmatrix}
W & Y(I_E \otimes Z) & Y(I_E \otimes X)(I_{E^{\otimes 2}} \otimes Z) & Y(I_E \otimes ((X^*)^{(2)})^*)(I_{E^{\otimes 3}} \otimes Z) & \cdots \\
0 & I_E \otimes W & I_E \otimes Y(I_E \otimes Z) & I_E \otimes Y(I_E \otimes X)(I_{E^{\otimes 2}} \otimes Z) & \cdots \\
0 & 0 & I_{E^{\otimes 2}} \otimes W & I_{E^{\otimes 2}} \otimes Y(I_E \otimes Z) & \cdots \\
\vdots & \vdots & \vdots & \vdots & \ddots 
\end{bmatrix}.
\end{equation*}
Note that $TU_\infty = V_\infty$.  We want to show that $T$ extends to an element of $\scrS(E,H, \sigma)$. It is easy to check that $T_{0j} \in \mathfrak{I}(\sigma^{E^{\otimes j}} \circ \varphi_j, \sigma), j \geq 0,$ because of the intertwining relations \eqref{eq: interrelations} satisfied by $X, Z, Y,$ and $W$.  To show that $\|T\| \leq 1$, we show that $T$ is the transfer map of a contractive time varying linear system.  

From the system of equations \eqref{eq: syseq} we have that, for all $t \in \N$ and for all $h \in H^{(N)}$,
\begin{equation} \label{eq: syseq2}
\begin{aligned}
(I_{E^{\otimes t}} \otimes L^*)\z^{(t)}h &= (I_{E^{\otimes t}} \otimes X)(I_{E^{\otimes t+1}} \otimes L^*)\z^{(t+1)}h + (I_{E^{\otimes t}} \otimes Z)(I_{E^{\otimes t}} \otimes U^*)\z^{(t)}h \\
(I_{E^{\otimes t}} \otimes V^*)\z^{(t)}h &= (I_{E^{\otimes t}} \otimes Y)(I_{E^{\otimes t+1}} \otimes L^*)\z^{(t+1)}h + (I_{E^{\otimes t}} \otimes W)(I_{E^{\otimes t}} \otimes U^*)\z^{(t)}h.
\end{aligned}
\end{equation}
Fix $h \in H^{(N)}$.  For $t \in \N$, define $x(t) = (I_{E^{\otimes t}} \otimes L^*)\z^{(t)}h, u(t) = (I_{E^{\otimes t}} \otimes U^*)\z^{(t)}h,$ and $y(t) = (I_{E^{\otimes t}} \otimes V^*)\z^{(t)}h.$ Also define $A(t) =  I_{E^{\otimes t}} \otimes X, B(t) = I_{E^{\otimes t}} \otimes Z, C(t) = I_{E^{\otimes t}} \otimes Y,$ and $D(t) = I_{E^{\otimes t}} \otimes W$.  Then 
\begin{eqnarray*}
U_\infty h = \begin{bmatrix} u(0) \\ u(1) \\ u(2) \\ \vdots \end{bmatrix},
V_\infty h = \begin{bmatrix} y(0) \\ y(1) \\ y(2) \\ \vdots \end{bmatrix},
\end{eqnarray*}
and the system \eqref{eq: syseq2} may be rewritten as 
\begin{eqnarray*}
x(t) &=& A(t)x(t+1) + B(t)u(t) \\
y(t) &=& C(t)x(t+1) + D(t)u(t), \quad t \in \N,
\end{eqnarray*}
where $\scrU(t) = E^{\otimes t} \otimes H, \scrY(t) = E^{\otimes t} \otimes H$, and $\scrH(t) = E^{\otimes t} \otimes H^{(N)}$.
Since $TU_\infty h = V_\infty h,$ $T$ is the transfer map of the system.  The matrices
\begin{eqnarray*}
\begin{bmatrix}
A(t) & B(t) \\
C(t) & D(t)
\end{bmatrix} = \begin{bmatrix}
I_{E^{\otimes t}} \otimes X & I_{E^{\otimes t}} \otimes Z \\
I_{E^{\otimes t}} \otimes Y & I_{E^{\otimes t}} \otimes W
\end{bmatrix} = I_{E^{\otimes t}} \otimes \begin{bmatrix} X & Z \\ Y & W \end{bmatrix}
\end{eqnarray*}
have norm equal to $1$ for all $t$, since $\Omega = \begin{bmatrix} X & Z \\ Y & W \end{bmatrix}$ is of norm $1$.  Thus $T$ is the transfer map of a contractive system.  Proposition \ref{thm: transmap} implies that $T \in \scrS(E,H,\sigma).$

Conversely, if there exists $T \in \scrS(E,H,\sigma)$ such that $TU_\infty = V_\infty$, then the solution to the displacement equation may be written as follows:
\begin{equation*}
A=U_\infty^*U_\infty - V_\infty^*V_\infty = U_\infty^* U_\infty - U_\infty^*T^*TU_\infty = U_\infty^*(I-T^*T)U_\infty \geq 0, 
\end{equation*}
since $\|T\| \leq 1$.
\end{proof}

Finally, we prove our generalized Nevanlinna-Pick Theorem.

\begin{proof}[Proof of Theorem \ref{thm: NPtheorem}]
We have already noted in Remark \ref{rk: pickmatrix} that the Pick matrix $A$ in equation \eqref{eq: pickmat} is the unique solution to the displacement equation, and we may write $A = U_\infty^*U_\infty - V_\infty^*V_\infty$.

If $A \geq 0$, then by Lemma \ref{thm: biglemma}, there exists $T \in \scrS(E,H,\sigma)$ such that $TU_\infty = V_\infty.$  By Remark \ref{rk: pickmatrix}, we rewrite $U_\infty$ and $V_\infty$ in terms of the Cauchy kernels to get
\begin{eqnarray*}
T \begin{bmatrix} C(\z_1) & \cdots & C(\z_N) \end{bmatrix} = \begin{bmatrix} (I_{\scrF(E)} \otimes \Lambda_1) C(\z_1) & \cdots & (I_{\scrF(E)} \otimes \Lambda_N) C(\z_N) \end{bmatrix}
\end{eqnarray*}
Comparing the matrices entrywise, we see that 
\begin{equation} \label{eq: entrywise}
TC(\z_i) = (I_{\scrF(E)} \otimes \Lambda_i) C(\z_i), \quad i = 1, \ldots, N.
\end{equation}
By Lemma \ref{thm: pteval}, we can rewrite the left hand side of equation \eqref{eq: entrywise} to get 
\begin{eqnarray*}
(I_{\scrF(E)} \otimes T(\z_i)) C(\z_i) = (I_{\scrF(E)} \otimes \Lambda_i) C(\z_i).
\end{eqnarray*}
It follows that $T(\z_i) = \Lambda_i$ for all $i = 1, \ldots, N$. Together with Lemma \ref{thm: equalsets}, this implies that there exists $X \in H^\infty(E^\sigma)$ with $\|X\| \leq 1$ such that $\hat{X}(\z_i) = \Lambda_i$ for all $i =1, \ldots, N$.

Conversely, suppose there exists $X \in H^\infty(E^\sigma)$ with $\|X\| \leq 1$ such that $\hat{X}(\z_i) = \Lambda_i$ for all $i =1, \ldots, N$.  Then by Lemma \ref{thm: equalsets}, there exists $T \in \scrS(E,H,\sigma)$ such that $T(\z_i) = \Lambda_i$ for all $i=1, \ldots, N$.  By the above calculations, $TU_\infty = V_\infty$, and by Lemma \ref{thm: biglemma}, $A \geq 0$. 
\end{proof}

\section{Remarks on Point Evaluations} 
\label{s: remarksonpteval}
As we noted after equation \eqref{eq: pteval}, the point evaluation defined by it is not multiplicative.  Nevertheless, thanks to \cite[Theorem 19]{Muhly2009}, point evaluations may be viewed as giving rise to antihomomorphisms from $H^\infty(E^\sigma)$ into the completely bounded maps on $\sigma(M)'$.  The proof of this assertion is simply a matter of shifting emphasis.  We let $E^\sigma$ play the role of $E$ in \cite[Theorem 19]{Muhly2009} and use Lemma \ref{thm: pteval}. Here are the details.

First recall that for a pair of points $\z$ and $\w$ in $E^\sigma$ with norm less than $1$, the inner product $\langle C(\z), C(\w) \rangle$ is given by the formula $\langle C(\z), C(\w) \rangle = C(\z)^*C(\w)$. Consequently, the map $a \mapsto \langle C(\z), \rho(\varphi_\infty^\sigma(a))C(\w) \rangle = \langle C(\z), (I_{\scrF(E)} \otimes a)C(\w) \rangle$ is a completely bounded map from $\sigma(M)'$ into itself.  So if $\z \in E^\sigma$ with $\| \z\|<1$ and $X \in H^\infty(E^\sigma)$, we may define the map $\Phi_X^\z$ on $\sigma(M)'$ by the formula 
\begin{equation} \label{eq: antihom}
\Phi_X^\z(a) :=\langle C(\z),\rho(\varphi_\infty^\sigma(a)) \rho(X) C(0) \rangle, \quad a \in \sigma(M)'.
\end{equation}
Note that $C(0) = \begin{bmatrix} I_H & 0 & 0 & \cdots \end{bmatrix}^T$, so our definition is precisely that in \cite[Theorem 19]{Muhly2009} with $E^\sigma$ replacing $E$, $\rho(\varphi_\infty^\sigma(a))$ replacing $\varphi_\infty(a)$, $\rho(X)$ replacing $X$, and $\z$ replacing $\xi$ in the notation of that theorem. 

Now note that formula (2) of \cite[Theorem 19]{Muhly2009} translated to our context allows us to write
\begin{equation*}
\rho(X)^*\rho(\varphi_\infty^\sigma(a^*))C(\z) = \rho(\varphi_\infty^\sigma(\Phi_X^\z(a)^*))C(\z).
\end{equation*}
Consequently, for $X, Z \in H^\infty(E^\sigma), a \in \sigma(M)',$ and $\z \in E^\sigma$ with $\| \z \|<1$,
\begin{multline*}
\Phi_{XZ}^\z(a) = \langle C(\z), \rho(\varphi_\infty^\sigma(a)) \rho(XZ)C(0) \rangle
=\langle \rho(X)^* \rho(\varphi_\infty^\sigma(a^*)) C(\z), \rho(Z) C(0) \rangle \\
= \langle \rho(\varphi_\infty^\sigma(\Phi_X^\z(a)^*))C(\z), \rho(Z)C(0) \rangle 
= \langle C(\z), \rho(\varphi_\infty^\sigma(\Phi_X^\z(a))) \rho(Z) C(0) \rangle
= \Phi_Z^\z(\Phi_X^\z(a)),
\end{multline*}
which shows that $\Phi_{XZ}^\z = \Phi_Z^\z \circ \Phi_X^\z$.  Thus we arrive at the following theorem.
\begin{theorem}
Let $\z \in E^\sigma$ with $\| \z \| <1$ and $X \in H^\infty(E^\sigma)$.  Define $\Phi_X^\z : \sigma(M)' \to \sigma(M)'$ by formula \eqref{eq: antihom}. Then the map $X \mapsto \Phi_X^\z$ is an algebra antihomomorphism from $H^\infty(E^\sigma)$ into the completely bounded maps on $\sigma(M)'$.  
\end{theorem}

We conclude with a theorem that relates our generalized Nevanlinna-Pick theorem, Theorem  \ref{thm: NPtheorem}, to Muhly and Solel's, Theorem \ref{thm: MSNPtheorem}, and gives a new characterization for interpolation in terms of completely bounded maps.  First, we define the center of a $W^*$-correspondence as in \cite[Definition 4.11]{Muhly2008b}.  

\begin{definition} If $E$ is a $W^*$-correspondence over a $W^*$-algebra M, then the \emph{center} of $E$, denoted $\mathfrak{Z}(E)$, is the collection of points $\xi \in E$ such that $a \cdot \xi = \xi \cdot a$ for all $a \in M$. 
\end{definition}

In \cite[Lemma 4.12]{Muhly2008b}, Muhly and Solel proved that if $E$ is a $W^*$-correspondence over a $W^*$-algebra $M$, then $\frakZ(E)$ is a $W^*$-correspondence over the commutative $W^*$-algebra $\frakZ(M)$. In general, we will say that a $W^*$-correspondence $E$ over a commutative $W^*$-algebra $M$ is \emph{central} if $E$ equals its center. Note that in Theorem \ref{thm: NPtheorem} (respectively, Theorem \ref{thm: MSNPtheorem}), the Pick matrix, and thus its positivity (resp., complete positivity), does not change if the correspondence $(\sigma(M)', E^\sigma)$ (resp., $(M, E)$) is replaced by the correspondence $(\frakZ(\sigma(M)'), \frakZ(E^\sigma))$ (resp., $(\frakZ(M), \frakZ(E))$).  Thus there exists an interpolating map in $H^\infty(E^\sigma)$ (resp., $H^\infty(E)$) if and only if there exists an interpolating map in $H^\infty(\frakZ(E^\sigma))$ (resp., $H^\infty(\frakZ(E))$).  Consequently, for the final theorem we restrict our attention to the correspondences $(\frakZ(\sigma(M)'), \frakZ(E^\sigma))$ and $(\frakZ(M), \frakZ(E))$.  We choose to do this because the centers are isomorphic as correspondences in the following sense.  

\begin{definition}[{\cite[Definition 2.2]{Muhly2008b}}]
An \emph{isomorphism} of a $W^*$-correspondence $E_1$ over $M_1$ and a $W^*$-correspondence $E_2$ over $M_2$ is a pair $(\sigma, \Psi)$ where $\sigma: M_1 \to M_2$ is an isomorphism of $W^*$-algebras, $\Psi:E_1 \to E_2$ is a vector space isomorphism, and for $e, f \in E_1$ and $a,b \in M_1$, we have $\Psi(a \cdot e \cdot b) = \sigma(a) \cdot \Psi(e) \cdot \sigma(b)$ and $\langle \Psi(e), \Psi(f) \rangle = \sigma( \langle e,f \rangle )$.
\end{definition}

Define $\gamma: \frakZ(E) \to \frakZ(E^\sigma)$ by $\gamma(\xi) = L_\xi$, where $L_\xi : H \to E \otimes_\sigma H$ is given by $L_\xi(h) = \xi \otimes h$. In \cite[Lemma 4.12]{Muhly2008b}, Muhly and Solel proved that the pair $(\sigma, \gamma)$ is an isomorphism of the correspondences $(\frakZ(M), \frakZ(E))$ and $(\frakZ(\sigma(M)'), \frakZ(E^\sigma))$.

\begin{proposition}
For $k \in \N$, define the map $\gamma_k : \frakZ(E)^{\otimes k} \to \frakZ(E^\sigma)^{\otimes k}$ by $\gamma_k(\xi_1 \otimes \cdots \otimes \xi_k) = L_{\xi_1} \otimes \cdots \otimes L_{\xi_k}$.  The pair $(\sigma, \gamma_k)$ is an isomorphism of $(\frakZ(M), \frakZ(E)^{\otimes k})$ onto $(\frakZ(\sigma(M)'), \frakZ(E^\sigma)^{\otimes k})$.
\end{proposition}

\begin{proof}
Let $\xi_1 \otimes \cdots \otimes \xi_k, \eta_1 \otimes \cdots \otimes \eta_k \in \frakZ(E)^{\otimes k}$.  Since $\xi_i, \eta_i \in \frakZ(E)$, $L_{\xi_j}, L_{\eta_i} \in \frakZ(E^\sigma)$, and $(\sigma, \gamma)$ is an isomorphism of correspondences, we have  
\begin{multline*}
\langle \gamma_k(\xi_1 \otimes \cdots \otimes \xi_k), \gamma_k(\eta_1 \otimes \cdots \otimes \eta_k) \rangle = \langle L_{\xi_1} \otimes \cdots \otimes L_{\xi_k}, L_{\eta_1} \otimes \cdots \otimes L_{\eta_k} \rangle \\
= \langle L_{\xi_2} \otimes \cdots \otimes L_{\xi_k}, \langle L_{\xi_1}, L_{\eta_1} \rangle \cdot L_{\eta_2} \otimes \cdots \otimes L_{\eta_k} \rangle \\
=\langle L_{\xi_2} \otimes \cdots \otimes L_{\xi_k}, L_{\eta_2} \otimes \cdots \otimes L_{\eta_k} \rangle \langle L_{\xi_1}, L_{\eta_1} \rangle 
= \ldots = \langle L_{\xi_k}, L_{\eta_k} \rangle \cdots \langle L_{\xi_1}, L_{\eta_1} \rangle \\
= \sigma(\langle \xi_k, \eta_k \rangle) \cdots \sigma(\langle \xi_1, \eta_1 \rangle)
= \sigma(\langle \xi_1 \otimes \cdots \otimes \xi_k, \eta_1 \otimes \cdots \otimes \eta_k \rangle).
\end{multline*}
Thus $\langle \gamma_k( \xi), \gamma_k(\eta) \rangle = \sigma( \langle \xi, \eta \rangle)$ for all $\xi, \eta \in \frakZ(E)^{\otimes k}$.  Furthermore, since $\sigma$ is an isomorphism, $\| \sigma(\langle \xi, \xi \rangle) \|= \| \langle \xi, \xi \rangle \|$ for all $\xi \in \frakZ(E)^{\otimes k}$.  Thus $\gamma_k$ is an isometry.

For $\xi_1 \otimes \xi_2 \otimes \cdots \otimes \xi_k \in \frakZ(E)^{\otimes k}, a, b \in \frakZ(M)$, and $h \in H$ we have
\begin{multline*}
\gamma_k(a \cdot (\xi_1 \otimes \xi_2 \otimes \cdots \otimes \xi_k) \cdot b)(h) =  \gamma_k((a \cdot \xi_1) \otimes \xi_2 \otimes \cdots \otimes (\xi_k \cdot b)) (h) = L_{a \cdot \xi_1} \otimes L_{\xi_2} \otimes \cdots \otimes L_{\xi_k \cdot b} (h) \\
=  L_{\xi_1}\sigma(a) \otimes L_{\xi_2} \otimes \cdots \otimes L_{\xi_k}\sigma(b) (h)
= L_{\xi_1} \otimes L_{\xi_2} \otimes \cdots \otimes L_{\xi_k}\sigma(a)\sigma(b) (h) \\
 = (I_{E^{\otimes k}} \otimes \sigma(a)) \gamma_k(\xi_1 \otimes \cdots \xi_k) \sigma(b) h.
\end{multline*}
Hence $\gamma_k(a \cdot \xi_1 \otimes \xi_2 \otimes \cdots \otimes \xi_k \cdot b) = \sigma(a) \cdot \gamma_k(\xi_1 \otimes \xi_2 \otimes \cdots \otimes \xi_k ) \cdot \sigma(b).$

\end{proof}

Define $\gamma_\infty : \scrF(\frakZ(E)) \to \scrF(\frakZ(E^\sigma))$ by $\gamma_\infty = diag[\sigma, \gamma, \gamma_2, \ldots]$.  Since $\gamma_k$ is an isomorphism of correspondences for each $k \in \N$, it follows that $\gamma_\infty$ is an isomorphism of correspondences as well. 

\begin{proposition}
For $\xi \in \frakZ(E)$, $\gamma_\infty T_\xi \gamma_\infty^{-1} = T_{\gamma(\xi)}$, where $T_\xi$ is the left creation operator in $H^\infty(\frakZ(E))$ determined by $\xi$, and $T_{\gamma(\xi)}$ is the left creation operator in $H^\infty(\frakZ(E^\sigma))$ determined by $\gamma(\xi)$. For $a \in \frakZ(M)$, $\gamma_\infty \varphi_\infty(a) \gamma_\infty^{-1} = \varphi_\infty^\sigma (\sigma(a))$, where $\varphi_\infty(a)$ is the left action operator in $H^\infty(\frakZ(E))$ determined by $a$, and $\varphi_\infty^\sigma (\sigma(a))$ is the left action operator in $H^\infty(\frakZ(E^\sigma))$ determined by $\sigma(a)$.
\end{proposition}
\begin{proof}
Let $\begin{bmatrix} \eta_0 & \eta_1 & \eta_2 & \cdots \end{bmatrix}^T \in \scrF(\frakZ(E^\sigma))$.  Since $\scrF(\frakZ(E^\sigma))$ is isomorphic to $\scrF(\frakZ(E))$ via $\gamma_\infty$, there exist $\alpha_{ij} \in \frakZ(E)$ such that $\eta_i = L_{\alpha_{i1}} \otimes \cdots \otimes L_{\alpha_{ii}},$ for $i \geq 1,$ and $\alpha_0 \in \frakZ(M)$ such that $\sigma(\alpha_0) = \eta_0$. Thus for $\xi \in \frakZ(E)$, we have
\begin{multline*}
\gamma_\infty T_\xi \gamma_\infty^{-1} \begin{bmatrix} \eta_0 \\ \eta_1 \\ \eta_2 \\ \vdots \end{bmatrix} = \gamma_\infty T_\xi \begin{bmatrix} \alpha_0 \\ \alpha_{11} \\ \alpha_{21} \otimes \alpha_{22} \\ \vdots \end{bmatrix}
= \gamma_\infty \begin{bmatrix} 0 \\ \xi \cdot \alpha_0 \\ \xi \otimes \alpha_{11} \\ \xi \otimes \alpha_{21} \otimes \alpha_{22} \\ \vdots \end{bmatrix} \\
= \begin{bmatrix} 0 \\ L_\xi \sigma(\alpha_0)   \\ L_\xi \otimes L_{\alpha_{11}} \\ L_\xi \otimes L_{\alpha_{21}} \otimes L_{\alpha_{22}} \\ \vdots \end{bmatrix} = T_{\gamma(\xi)} \begin{bmatrix} \eta_0 \\ \eta_1 \\ \eta_2 \\ \vdots \end{bmatrix}.
\end{multline*}
For $a \in \frakZ(M)$, we have
\begin{multline*}
\gamma_\infty \varphi_\infty(a) \gamma_\infty^{-1} \begin{bmatrix} \eta_0 \\ \eta_1 \\ \eta_2 \\ \vdots \end{bmatrix} = \gamma_\infty \varphi_\infty(a)  \begin{bmatrix} \alpha_0 \\ \alpha_{11} \\ \alpha_{21} \otimes \alpha_{22} \\ \vdots \end{bmatrix} = \gamma_\infty \begin{bmatrix} a \alpha_0 \\ \varphi(a) (\alpha_{11}) \\  \varphi_2(a) (\alpha_{21} \otimes \alpha_{22}) \\ \vdots \end{bmatrix} \\
 = \begin{bmatrix} \sigma(a \alpha_0) \\ L_{a \cdot \alpha_{11}} \\ L_{a \cdot \alpha_{21}} \otimes L_{\alpha_{22}} \\ \vdots \end{bmatrix} 
= \begin{bmatrix} \sigma(a) \sigma(\alpha_0) \\ \sigma(a) \cdot L_{\alpha_{11}} \\ \sigma(a) \cdot L_{\alpha_{21}} \otimes L_{\alpha_{22}} \\ \vdots \end{bmatrix} = \varphi_\infty^\sigma(\sigma(a)) \begin{bmatrix} \eta_0 \\ \eta_1 \\ \eta_2 \\ \vdots \end{bmatrix}.
\end{multline*}
\end{proof}

Thus we arrive at the following isomorphism from $H^\infty(\frakZ(E))$ onto $H^\infty(\frakZ(E^\sigma))$.

\begin{proposition}
The map defined on the generators of $H^\infty(\frakZ(E))$ by $T_\xi \mapsto \gamma_\infty T_\xi \gamma_\infty^{-1}, \xi \in \frakZ(E),$ and $\varphi_\infty(a) \mapsto \gamma_\infty \varphi_\infty(a) \gamma_\infty^{-1}, a \in \frakZ(M)$, extends to an isomorphism  $\Gamma$ from $H^\infty(\frakZ(E))$ onto $H^\infty(\frakZ(E^\sigma))$.  
\end{proposition}

With $\Gamma$ in hand, we are finally able to prove that Theorems \ref{thm: NPtheorem} and \ref{thm: MSNPtheorem} are essentially the same when we restrict to the centers of the correspondences.

\begin{theorem} \label{thm: connection}
Let $\z_1, \ldots, \z_N \in \mathfrak{Z}(E^\sigma)$ with $\|\z_i \| <1, i=1, \ldots, N,$ and $\Lambda_1, \ldots, \Lambda_N \in \mathfrak{Z}(\sigma(M)')$.  Define $\Psi_{\Lambda_i}^{\z_i} : \sigma(M)' \to \sigma(M)'$ by $\Psi_{\Lambda_i}^{\z_i}(a) = \langle C(\z_i), (I_{\scrF(E)} \otimes a\Lambda_i^*)C(0) \rangle$. For $X \in H^\infty(\frakZ(E^\sigma))$, define $\Phi_X^{\z_i} : \sigma(M)' \to \sigma(M)'$ by $\Phi_X^{\z_i}(a) =\langle C(\z_i),\rho(\varphi_\infty^\sigma(a)) \rho(X) C(0) \rangle$. Then the following are equivalent:
\begin{enumerate}
\item there exists $Y \in H^\infty(\frakZ(E))$ with $\|Y \| \leq 1$ such that $\hat{Y}(\z_i^*) = \Lambda_i^*, i=1, \ldots, N,$ in the sense of Theorem \ref{thm: MSNPtheorem}.
\item there exists $X=\Gamma(Y) \in H^\infty(\frakZ(E^\sigma))$ with $\|X \| \leq 1$ such that $\hat{X}(\z_i) = \Lambda_i, i = 1, \ldots, N$, in the sense of Theorem \ref{thm: NPtheorem}.
\item $\Phi_X^{\z_i} = \Psi_{\Lambda_i}^{\z_i}, i=1, \ldots, N$.
\end{enumerate}
\end{theorem}

\begin{proof}
First we prove (1) is equivalent to (2).  Suppose $\xi \in \frakZ(E)$ and $Y = T_\xi \in H^\infty(\frakZ(E))$. By definition of point evaluation in Theorem \ref{thm: MSNPtheorem}, we have
\begin{eqnarray*}
\hat{Y}(\z_i^*)^* = \langle(Y \otimes I_H)C(0), C(\z_i) \rangle = \langle (T_\xi \otimes I_H) C(0), C(\z_i) \rangle
= ({T_\xi^{(1)}}^* \otimes I_H) \z_i,
\end{eqnarray*}
where $T_\xi^{(1)}$ is defined after equation \eqref{eq: creationop}.
Since $\frakZ(E)$ is isomorphic to $\frakZ(E^\sigma)$ via $\gamma,$ there exists $\alpha_i \in \frakZ(E)$ such that $\z_i = L_{\alpha_i}$.  Thus for $h \in H$,
\begin{eqnarray*}
({T_\xi^{(1)}}^* \otimes I_H) \z_i h = {T_\xi^{(1)}}^* (\alpha_i) \otimes h = \langle \xi, \alpha_i \rangle \otimes h= \sigma(\langle \xi, \alpha_i \rangle)h.
\end{eqnarray*}
Now for $X= T_{\gamma(\xi)} \in H^\infty(\frakZ(E^\sigma))$, by definition of point evaluation in equation \eqref{eq: pteval} we see that 
\begin{multline*}
\hat{X}(\z_i)h = \langle U(X \otimes I_H) U^* C(0), C(\z_i) \rangle h = \langle (T_{\gamma(\xi)} \otimes I_H) U^*C(0), U^*C(\z_i) \rangle h \\
 = ({T_{\gamma(\xi)}^{(1)}}^* \otimes I_H) L_{\z_i} h = {T_{\gamma(\xi)}^{(1)}}^*(\z_i) \otimes h = \langle \gamma(\xi), \z_i \rangle \otimes h = L_\xi^* (\alpha_i \otimes h)  \\
= \sigma(\langle \xi, \alpha_i \rangle) h= \hat{Y}(\z_i^*)^* h.
\end{multline*}
Thus $\hat{Y}(\z_i^*) = \Lambda_i^*$ if and only if $\hat{X}(\z_i)= \Lambda_i$, where $X = \Gamma(Y) \in H^\infty(\frakZ(E^\sigma))$.  Similarly, if $a \in \frakZ(M)$ and $Y = \varphi_\infty(a) \in H^\infty(\frakZ(E))$, then a quick calculation shows that $\hat{Y}(\z_i^*)^* = \sigma(a)^*$.  Moreover, if $X = \varphi_\infty^\sigma(\sigma(a)) \in H^\infty(\frakZ(E^\sigma))$, then $\hat{X}(\z_i) = \sigma(a)^*$ as well.  So again we have that $\hat{Y}(\z_i^*) = \Lambda_i^*$ if and only if $\hat{X}(\z_i)= \Lambda_i$, where $X = \Gamma(Y) \in H^\infty(\frakZ(E^\sigma))$.  Since the equivalence holds for the generators of $H^\infty(\frakZ(E))$ and  $H^\infty(\frakZ(E^\sigma))$, we conclude that (1) is equivalent to (2).

The following calculation shows that (2) and (3) are equivalent. Note that the proof of Lemma \ref{thm: pteval} shows that for all $X \in H^\infty(E^\sigma)$, all $\z \in \mathfrak{Z}(E^\sigma)$ with $\| \z \|<1$, and all $a \in \sigma(M)'$,
\begin{equation*}
\rho(X)^*\rho(\varphi_\infty^\sigma(a^*))C(\z) = \rho(X)^*(I_{\scrF(E)} \otimes a^*)C(\z) = (I_{\scrF(E)} \otimes \hat{X}(\z))(I_{\scrF(E)} \otimes a^*)C(\z).
\end{equation*} 
Thus we have
\begin{multline*}
\Phi_X^{\z_i}(a) = \langle C(\z_i), \rho(\varphi_\infty^\sigma(a)) \rho(X)C(0) \rangle = \langle \rho(X)^* \rho(\varphi_\infty^\sigma(a^*)) C(\z_i), C(0) \rangle \\
= \langle(I_{\scrF(E)} \otimes \hat{X}(\z_i))(I_{\scrF(E)} \otimes a^*)C(\z_i), C(0) \rangle = \langle C(\z_i), (I_{\scrF(E)} \otimes a\hat{X}(\z_i)^* C(0) \rangle, 
\end{multline*}
which is equal to $\Psi_{\Lambda_i}^{\z_i}(a)$ if and only if $\hat{X}(\z_i) = \Lambda_i$ for $i=1, \ldots, N$.
\end{proof}

We note that in \cite{Ball2010}, Ball and ter Horst also addressed the issues we discussed here.  Their context is not quite as general as ours and it is formulated somewhat differently, but it may be possible to extend their arguments to our setting.  The first condition in \cite[Theorem 3.3]{Ball2010} seems related to Muhly and Solel's theorem via an application of Choi's criterion \cite[Theorem 2]{Choi1975}.  Consequently, the relationship between Muhly and Solel's Theorem 5.3 in \cite{Muhly2004a} and Constantinescu and Johnson's Theorem 3.4 in \cite{Constantinescu2003} is illuminated further by their theorem.

\section*{acknowledgements}
I would like to thank Joe Ball for his insight into the connection between Constantinescu and Johnson's result and Muhly and Solel's; Jenni Good for her meticulous explanations and observations; and Paul Muhly and Baruch Solel for their astute feedback.

\bibliographystyle{plain}
\bibliography{masterbibliography}

\def\cprime{$'$}
\begin{thebibliography}{10}

\bibitem{Ball2010}
Joseph~A. Ball and Sanne ter Horst.
\newblock {Multivariable operator-valued Nevanlinna-Pick interpolation: a
  survey}.
\newblock In J.~J. Grobler, L.~E. Labuschagne, and M.~M{\"o}ller, editors, {\em
  {Operator Algebras, Operator Theory and Applications: 18th International
  Workshop on Operator Theory and Applications, Potchefstroom, July 2007}},
  pages 1--72. Birkh{\"a}user, Basel, 2010.

\bibitem{Choi1975}
Man-Duen Choi.
\newblock Completely positive linear maps on complex matrices.
\newblock {\em Linear Algebra and its Applications}, 10(3):285 -- 290, 1975.

\bibitem{Constantinescu2003}
T.~{Constantinescu} and J.~L. {Johnson}.
\newblock {A note on noncommutative interpolation}.
\newblock {\em Canad. Math. Bull.}, 46(1):59--70, 2003.

\bibitem{Constantinescu1996}
Tiberiu Constantinescu.
\newblock {\em Schur parameters, factorization and dilation problems},
  volume~82 of {\em Operator Theory: Advances and Applications}.
\newblock Birkh\"auser Verlag, Basel, 1996.

\bibitem{Douglas1966a}
R.~G. Douglas.
\newblock On majorization, factorization, and range inclusion of operators on
  {H}ilbert space.
\newblock {\em Proc. Amer. Math. Soc.}, 17:413--415, 1966.

\bibitem{Kailath1979}
Thomas Kailath, Sun-Yuan Kung, and Martin Morf.
\newblock Displacement ranks of matrices and linear equations.
\newblock {\em Journal of Mathematical Analysis and Applications}, 68(2):395 --
  407, 1979.

\bibitem{Kailath1995}
Thomas Kailath and Ali~H. Sayed.
\newblock Displacement structure: Theory and applications.
\newblock {\em SIAM Review}, 37(3):297--386, 1995.

\bibitem{Muhly2002a}
Paul~S. Muhly and Baruch Solel.
\newblock {Quantum {M}arkov processes (correspondences and dilations)}.
\newblock {\em Internat. J. Math.}, 13(8):863--906, 2002.

\bibitem{Muhly2004a}
Paul~S. Muhly and Baruch Solel.
\newblock {Hardy algebras, {$W^\ast$}-correspondences and interpolation
  theory}.
\newblock {\em Math. Ann.}, 330(2):353--415, 2004.

\bibitem{Muhly2008b}
Paul~S. Muhly and Baruch Solel.
\newblock {Schur class operator functions and automorphisms of {H}ardy
  algebras}.
\newblock {\em Doc. Math.}, 13:365--411, 2008.

\bibitem{Muhly2009}
Paul~S. Muhly and Baruch Solel.
\newblock {The {P}oisson kernel for {H}ardy algebras}.
\newblock {\em Complex Anal. Oper. Theory}, 3(1):221--242, 2009.

\end{thebibliography}
%\bibliography{/Users/Paul_Muhly/Dropbox/BP_Share/Master_Bib_File/Master130901-1}

\end{document}